\documentclass[11pt]{amsart}
\usepackage{amscd,amssymb,amsthm,amsmath,amssymb,mathrsfs,enumerate,enumitem,xfrac,txfonts}
\usepackage[matrix,arrow,curve]{xy}
\usepackage[margin=2cm]{geometry}
\usepackage{xcolor}
\usepackage{comment}
\usepackage[colorlinks=true, linkcolor=blue!70!black, urlcolor=purple, citecolor=blue!70!black]{hyperref}

\newtheorem{theorem}{Theorem}[section]
\newtheorem{proposition}[theorem]{Proposition}
\newtheorem{lemma}[theorem]{Lemma}
\newtheorem*{lemma*}{Lemma}
\newtheorem{corollary}[theorem]{Corollary}
\newtheorem{conjecture}[theorem]{Conjecture}

\newtheorem{problem}[theorem]{Problem}
\newtheorem*{maintheorem*}{Main Theorem}
\newtheorem*{corollary*}{Corollary}
\newtheorem*{conjecture*}{Conjecture}
\newtheorem*{theorem*}{Theorem}
\newtheorem{claim}{Claim}
\newtheorem*{question*}{Question}
\newtheorem*{proposition*}{Proposition}
\newtheorem*{remark*}{Remark}
\newtheorem*{maincorollary*}{Main Corollary}
\newtheorem{setup}[theorem]{Setup}

\theoremstyle{definition}
\newtheorem{example}[theorem]{Example}
\newtheorem*{example*}{Example}
\newtheorem{definition}[theorem]{Definition}
\newtheorem*{acknowledgement}{Acknowledgments}

\theoremstyle{remark}

\theoremstyle{plain}

\theoremstyle{definition}

\newcommand{\bC}{\mathbb{C}}

\newcommand{\bQ}{\mathbb{Q}}

\newcommand{\bF}{\mathbb{F}}
\newcommand{\bG}{\mathbb{G}}

\newcommand{\calL}{\mathcal{L}}

\newcommand{\calX}{\mathcal{X}}

\newcommand{\Proj}{\mathrm{Proj}}

\newcommand{\Pic}{\mathrm{Pic}}
\newcommand{\Aut}{\mathrm{Aut}}

\newcommand{\Spec}{\mathrm{Spec}\;}

\newcommand{\PP}{\mathbb{P}}
\newcommand{\bP}{\mathbb{P}}
\newcommand{\bA}{\mathbb{A}}

\newcommand{\SL}{\mathrm{SL}}

\newcommand{\calO}{\mathcal{O}}
\newcommand{\cO}{\mathcal{O}}

\newcommand{\calM}{\mathcal{M}}

\usepackage{graphicx}
\newcommand{\bZ}{\mathbb{Z}}

\newcommand{\Chow}{\operatorname{Chow}}

\newcommand{\cX}{\mathcal X}
\newcommand{\cD}{\mathcal D}

\newcommand{\CM}{\mathrm{CM}}
\newcommand{\CH}{\mathrm{CH}}

\newcommand{\cL}{\mathcal L}

\newcommand{\Fut}{\mathrm{Fut}}

\newcommand{\wt}{\mathrm{wt}}

\def\tiago#1{{\color{blue}\textbf{Tiago: }#1}}
\def\livia#1{{\color{teal}\textbf{Livia: }#1}}
\def\kristin#1{{\color{violet}\textbf{Kristin: }#1}}

\makeatletter\@addtoreset{equation}{section} \makeatother

\author{Livia Campo, 
Kristin DeVleming, Tiago Duarte Guerreiro}

\address{Institut f\"{u}r Mathematik \\
	Universit\"{a}t Wien \\
	Oskar-Morgenstern-Platz 1 \\
	1090 Wien \\
	Austria}
	
\email{livia.campo@univie.ac.at}

\address{Department of Mathematics \\
University of California, San Diego \\
2985 Muir Lane \\
La Jolla, CA  92093 \\
USA
}
	
\email{kedevleming@ucsd.edu}

\address{Departement Mathematik und Informatik (DMI) \\ Universit\"{a}t Basel \\
Spiegelgasse 1 \\
4051 Basel \\
Switzerland}
	
\email{tiago.duarteguerreiro@unibas.ch}

\title{K-stability of complete intersections}

\begin{document}

\begin{abstract}
    We prove the K-polystability of the general Fano complete intersection of arbitrary multidegree and dimension, and the K-stability of the general Fano complete intersection that is not isomorphic to projective space or a quadric hypersurface.  We prove analogous results for certain smooth weighted complete intersections.
\end{abstract}

\maketitle

\vspace{-.5cm}

\section{Introduction}

K-stability is an algebraic notion introduced in \cite{Tia97, Don02} to capture the existence of a \textit{K\"ahler-Einstein} metric, a type of constant scalar curvature metric, on Fano manifolds.  Indeed, by \cite{CDS1} and \cite{Tian15}, a smooth Fano manifold $X$ is K-polystable if and only if $X$ admits a K\"ahler-Einstein metric.  This result has been generalized to klt Fano varieties by \cite{LXZ22}. 

There have been great strides in K-stability theory in algebraic geometry over the past several decades, notably the construction of projective moduli spaces parameterizing K-polystable Fano varieties of fixed dimension and anticanonical volume.  However, in general, determining the K-(poly/semi)stability of a particular Fano manifold remains difficult.  In dimension 2, the situation is completely understood: there are ten deformation families of del Pezzo surfaces, and the K-stability of every member of each family has been determined: $\bP^2$, $\bP^1 \times \bP^1$, and smooth del Pezzo surfaces of degree 6 are K-polystable; the first Hirzebruch surface $\bF_1$ and del Pezzo surfaces of degree 7 are K-unstable; and smooth del Pezzo surfaces of degree $\le 5$ are K-stable.  In dimension 3, by the remarkable work \cite{thebook:CalabiProb} the K-(poly/semi)stability of the \textit{general} member of each deformation family of smooth Fano threefolds is known.  For most families, the K-stability of \textit{all} smooth members is difficult to determine but is near completion. In higher dimensions, much less is known.

The present paper investigates the K-stability of Fano hypersurfaces and complete intersections.  Denoting by $X$ the Fermat hypersurface of degree $d \le n$ in $\bP^n$, there are several methods to prove $X$ is K-polystable and in fact K-stable if $d \ge 3$.  For $d = 1,2$, as every smooth hypersurface is isomorphic to the Fermat, this proves all such $X$ are K-polystable. For $d \ge 3$, given a family of Fano varieties over a smooth base $B$, the set of K-(semi)stable fibers of the family is an open subset of $B$ by \cite{BLX19}.  Therefore, if $B = \bP(H^0(\bP^n, \calO(d)))$ is the parameter space of degree $d$ hypersurfaces in $\bP^n$ and $\calX \to B$ is the universal family, the K-stability of the Fermat hypersurface then implies the K-stability of the general hypersurface, which leads to the following folklore conjecture: 

\begin{conjecture}
	Every smooth Fano hypersurface of degree $d$ in $\bP^{n+1}$ is K-polystable, and K-stable if $d \ge 3$.
\end{conjecture}

This conjecture is known in very few cases.  By the previous paragraph, this is known for hypersurfaces of degree $d = 1,2$.   By \cite{LiuXu19,liu2022cubicfourfolds}, the result follows for cubic hypersurfaces in $\bP^{n+1}$ for $n \le 4$. In \cite{kentohyper}, Fujita proves the result for hypersurfaces of Fano index $1$, i.e. hypersurfaces of degree $n+1$ in $\bP^{n+1}$. Using a powerful new method to estimate the stability threshold, Abban and Zhuang showed in \cite{AZ21} that every smooth hypersurface of Fano index $2$, i.e. degree $n$ in $\mathbb P^{n+1}$, is K-stable for $n \ge 3$, and, building on their previous work, prove in \cite{AZ-Seshadri} that any smooth hypersurface $X$ in $\mathbb P^{n+1}$ such that $n\geq \iota_X^3\geq 9$ is K-stable, where $\iota_X$ is the Fano index of $X$.

The same folklore conjecture is made for complete intersections: 

\begin{conjecture}
	Assume $k \ge 2$ and $r_1, \dots, r_k \ge 2$.  Every smooth Fano complete intersection $X$ of multidegree $(r_1, \dots, r_k)$ in $\bP^{n+k}$ is K-stable.
\end{conjecture}

To the authors' knowledge, this result is known only in the case that $X$ is the intersection of two quadrics (\cite{AGP06}), the complete intersection of a quadric and cubic in $\mathbb P^5$ (\cite[Proposition 5.3]{ZhuangSR}), the complete intersection of three quadrics in $\bP^6$ (\cite[Theorem 5.1]{AZ-Seshadri}), and complete intersections of Fano index one of dimension $n$ in $\mathbb P^{n+k}$ such that $n\geq 10k$ (\cite[Theorem~1.3]{ZhuangSR}).

In general, there is remarkably little evidence for this conjecture.  The main result of this paper is the following theorem, providing evidence for the conjecture in all cases and giving a complete answer to \cite[Problem~8]{xu2026openproblemskstabilityfano} proposed by Xu and Zhuang:

\begin{theorem} \label{thm: index >=2}
   	If $X \subset \mathbb{P}^{n+k}$ is a general $n$-dimensional Fano complete intersection of type $(r_1, \dots, r_k)$, then $X$ is K-polystable.  If $X$ is not isomorphic to $\bP^n$ or a smooth quadric, then $X$ is K-stable.
\end{theorem}

As the anticanonical volume of a complete intersection as above is $v = (n+k+1 - (r_1 + \dots + r_k))^n r_1 \dots r_k$ and small deformations of non-K3 complete intersections are complete intersections (\cite{Sernesi75}), this immediately gives: 

\begin{corollary}
	For any $n \ge 2$ and positive integers $k,r_1, \dots, r_k$ such that $n+k+1 > r_1 + \dots + r_k$, for $v =  (n+k+1 - (r_1 + \dots + r_k))^n r_1 \dots r_k$, the K-moduli space $M^K_{n,v}$ is non-empty and contains an irreducible component with generic point parameterizing a smooth complete intersection of type $(r_1, \dots, r_k)$ in $\bP^{n+k}$.
\end{corollary}

Here, $M^K_{n,v}$ is the K-moduli space of K-polystable Fano varieties of dimension $n$ and volume $v$ (cf. \cite{Chenyang}). A complete explicit description of the corresponding component of the K-moduli space is unknown except in few cases where K-stability is known to coincide with GIT stability, see e.g. \cite{OSS, SS17, LiuXu19, liu2022cubicfourfolds}.  In general, the K-moduli space does not coincide with GIT stability and very little is known about the boundary of these spaces. 

In fact, we prove a more general result than Theorem \ref{thm: index >=2} for weighted complete intersections. 

\begin{theorem} \label{thm:wted-version}
   	If $X \subset \mathbb{P}(1^{n+1},a_{n+1}, \dots, a_{n+k})$ is a general $n$-dimensional Fano weighted complete intersection of type $(r_1, \dots, r_k)$ such that $a_{n+i} \mid r_i$ for each $i$, then $X$ is K-polystable.  If $X$ is not isomorphic to $\bP^n$ or a smooth quadric, then $X$ is K-stable.
\end{theorem}

To prove the main theorems, using the openness of K-stability in families, it suffices to find one K-stable complete intersection of each multidegree.  To do so, we prove several results on the stability of complete intersections of particular forms.  

\begin{theorem}[See Theorem \ref{thm:index1}]
	Let $k \ge 1$ and $n \ge 2$.  Let $X\subset \mathbb P(1^{n+1},a_{n+1},\ldots,a_{n+k})$ be a Fano weighted complete intersection of type $(r_1,\ldots,r_k)$, where $r_i \ge 2$ and  $a_{n+i} \, \vert \, r_i$ for each $i$, with the following form: 
	\begin{equation*}
		X\colon \begin{cases}
			f_1+x_{n+1}^{r_1/{a_{n+1}}}=0\\
			f_2+x_{n+2}^{r_2/{a_{n+2}}}=0\\
			\quad \vdots \\
			f_k+x_{n+k}^{r_k/{a_{n+k}}}=0\\
		\end{cases} \subset \mathbb P(1^{n+1},a_{n+1},\ldots,a_{n+k})
	\end{equation*}
	such that $f_i \in \mathbb C[x_0,\ldots, x_n]$ is a homogeneous polynomial of degree $r_i$ for $1\leq i\leq k $ and  $\sum B_i$ is snc where $B_i$ is the divisor in $\mathbb P^n$ corresponding to the vanishing of $f_i$. Let $q:=\min\Big\{\frac{a_{n+1}}{r_1},\ldots,\frac{a_{n+k}}{r_k}\Big\}$. 
    \begin{enumerate}
        \item If the Fano index $\iota_X$ satisfies
	\begin{equation*} 
	\iota_X < (n+1)q    
	\end{equation*}
	then $X$ is K-polystable and is K-stable if $\mathrm{Aut}(X)$ is finite.
    \item If the Fano index $\iota_X$ satisfies
	\begin{equation*} 
	\iota_X = (n+1)q    
	\end{equation*}
	and each $B_i$ is a GIT-polystable hypersurface of $\mathbb P^n$, then $X$ is K-polystable and is K-stable if $\mathrm{Aut}(X)$ is finite.
    \end{enumerate}
In particular, if $X$ as above is smooth such that $\iota_X\leq (n+1)q$ and is not isomorphic to $\mathbb P^n$ or a quadric, then it is K-stable.
\end{theorem}

This implies the K-stability of the general Fano complete intersection of small index; in particular, the general complete intersection of index one will be K-polystable (see Corollary \ref{cor:index-one}). 

Without the condition on the Fano index, we have the following: 

\begin{theorem}[See Theorem \ref{thm:general-case-induction}]
	    For integers $n \geq 2$ and $k \geq 2$, let $X \subset \mathbb{P}(1^{n+1}, a_{n+1}, \dots, a_{n+k})$ be a Fano weighted complete intersection of type $(r_1, \dots, r_k)$ such that $a_{n+i} \mid r_i$ and dimension $\dim X=n$ with equations
	\begin{equation*}
		X\colon \begin{cases}
			f_1=0\\
			f_2=0\\
			\quad \vdots \\
			f_{k-1}=0 \\
			f_k+x_{n+k}^{r_k/a_{n+k}}=0\\
		\end{cases} 
	\end{equation*}
	where $f_1, \dots ,f_k \in \mathbb{C} \left[ x_0, \dots , x_{n+k-1} \right]$ are a regular sequence of homogeneous polynomials of degree $\deg f_i = r_i$. 
	Call $Y \subset \mathbb{P}(1^{n+1}, a_{n+1}, \dots, a_{n+k-1})$ the complete intersection of type $(r_1, \dots, r_{k-1})$ defined by 
	\begin{equation*}
		Y \colon \begin{cases}
			f_1=0\\
			f_2=0\\
			\quad \vdots \\
			f_{k-1}=0 \\
		\end{cases} 
	\end{equation*}
	Call $B_Y$ the complete intersection $Y \cap (f_k=0) \subset \mathbb{P}(1^{n+1}, a_{n+1}, \dots, a_{n+k-1})$.  Assume $Y$ is K-polystable. If $B_Y$ is Fano, assume that $B_Y$ is K-polystable, and if $B_Y$ is not Fano, assume that the pair $(Y, \frac{n+1+\sum_{i=1}^{k-1} a_{n+i}-\sum_{i=1}^{k-1} r_i}{r_k}B_Y)$ is log canonical. Then, $X$ is K-polystable. If $\Aut(X)$ is finite, then $X$ is K-stable.
\end{theorem}

The previous result allows us to inductively show the K-stability of the general complete intersection.  Indeed, assuming the $f_i$ are general, we may assume $X, Y, B_Y$ are smooth and inductively we may suppose the assumptions on $Y$ and $B_Y$ hold.  We then conclude $X$ is K-polystable.  When $a_{n+1} = \dots = a_{n+k} = 1$, by \cite{LW84} or \cite{benoist2013separation}, the automorphism group of $X$ is finite, so we conclude that a general Fano complete intersection is K-stable, which implies Theorem \ref{thm: index >=2}.  

While the results of the previous two theorems overlap, we include both results as their applications are distinct.  For example, the first result allows us to write down many examples of Fano complete intersections of small index that are K-stable by simply writing polynomials $B_i = (f_i = 0)$ such that $\sum B_i$ is snc.  The second result is more suitable for inductive proofs: use this in explicit examples would require first verifying polystability properties of the varieties $Y$ and $B_Y$.  The proofs also proceed along different methods.   

We also derive several additional corollaries on K-stability of weighted hypersurfaces and complete intersections of particular forms.  First, we provide an algebraic proof of the following result in \cite{AGP06}. See also \cite[Corollary~4.17]{ZhuangEquivariant}.

\begin{corollary}[See Corollary \ref{cor:int-of-2-quadrics}]
    Let $X \subset \bP^{n+2}$ be any smooth Fano complete intersection of two quadrics, $n \ge 2$.  Then, $X$ is K-stable.
\end{corollary}

We also prove the following for intersections of three quadrics: 

\begin{theorem}[See Theorem \ref{three-quadrics}]
	Let $X$ be smooth complete intersection of three quadrics $X = Q_1 \cap Q_2 \cap Q_3 \subset \bP^7$ such that $Q_1$ and $Q_2$ have a common singular point.  Then, $X$ is K-stable.
\end{theorem}

Several additional corollaries appear in Section \ref{sec:covers}.

\begin{acknowledgement}
    We thank Ivan Cheltsov, Elizabeth Gross, and Julius Ross for organizing the 2025 working group on \textit{Polarized varieties and their applications} at which this collaboration began.  We also thank Hamid Abban, Ivan Cheltsov, Kento Fujita, Yijue Hu, and Patricio Gallardo for helpful conversations.  During the preparation of this manuscript, we were informed that Yijue Hu obtained the same results on K-stability of complete intersections of Fano index one, which will appear in forthcoming work. LD was partially supported by \"{O}FG International Communications Grant No. IK-1446. KD was supported by NSF grant DMS-2550445 (formerly DMS-2302163). TDG was supported by ERC StG Saphidir No. 101076412. 
\end{acknowledgement}

\section{Properties of hypersurfaces and complete intersections}

In what follows, we will use the following notation.  A \textbf{hypersurface} of degree $d$ in $\bP^{n+1}$ will be denoted by $X_d \subset \bP^{n+1}$.  The variety $X_d$ has $\dim X_d = n$ and the canonical divisor $X_d$ is given by $\calO(K_{X_d}) \cong \calO_{\bP^{n+1}}(-n-2+d)|_{X_d}$.  Observe that $X_d$ is Fano if and only if $d \le n+1$.  

A \textbf{complete intersection} of multidegree $(r_1, \dots, r_k)$ in $\bP^{n+k}$ is a variety of dimension $n$ denoted $X_{r_1, \dots,r_k} \subset \bP^{n+k}$ defined by the vanishing of a regular sequence of $k$ polynomials $f_1, \dots, f_k \in \bC[x_0, \dots, x_{n+k}]$ such that each $f_i$ is a homogeneous polynomial of degree $r_i$.  The canonical divisor of the variety $X_{r_1, \dots,r_k}$ is given by $\calO(K_{X_{r_1,\dots,r_k}}) \cong \calO_{\bP^{n+1}}(-n-k-1+r_1+\dots+r_k)|_{X_{r_1,\dots,r_k}}$ and $X_{r_1, \dots,r_k}$ is Fano if and only if $r_1 + \dots + r_k \le n+k$.

\begin{definition}
    Let $X$ be a Fano variety such that $\calO_X(K_X)$ is Cartier.  The \textbf{Fano index} of $X$ is
    \[ \iota_X = \max_{m \in \bZ \ge 0} \{ -K_X \sim mH \mid H \in \Pic(X) \}.\]

    For a Fano hypersurface $X_d\subset \bP^{n+1}$, the Fano index is $\iota_{X_d} = n+2 - d$, and for a Fano complete intersection $X_{r_1, \dots, r_k} \subset \bP^{n+k}$, the Fano index is $ \iota_{X_{r_1, \dots, r_k}} = n+k+1 - r_1 - \dots - r_k$.
\end{definition}

In most cases, the automorphism groups of smooth hypersurfaces and complete intersections are finite: 

\begin{theorem}\cite{MM63}
    Let $X_d \subset \bP^{n+1}$ be a smooth degree $d \ge 2$ hypersurface of dimension $\ge 2$ that is not a K3 surface or a quadric hypersurface.  Then, $\Aut(X_d)$ is finite.
\end{theorem}

\begin{theorem}\cite{LW84,benoist2013separation}
    Let $X_{r_1,\dots,r_k} \subset \bP^{n+k}$ be a smooth complete intersection of dimension $\ge 2$ of multidegree $(r_1, \dots, r_k)$ with $r_i \ge 2$ for all $i$ that is not a K3 surface or a quadric.  Then, $\Aut(X_{r_1,\dots,r_k})$ is finite.
\end{theorem}

There are analogous definitions and results for weighted hypersurfaces and complete intersections in weighted projective space.  Let $(a_0, \dots, a_{n})$ be a collection of positive integers such that no subset of size $n$ has a common factor.  Define $\bP(a_0, \dots, a_n) = \Proj \  \bC[x_0, \dots, x_n]$ where the grading on the polynomial ring is such that $\deg x_i = a_i$.  We will use superscripts on the weights to indicate a repeated weight, e.g. $\bP(1^3,2) = \bP(1,1,1,2)$.  This is the quotient of the affine space $(\Spec \bC[x_0, \dots, x_n] \setminus \{0\})/\bC^* $ where $\bC^*$ acts on $x_i$ with weight $a_i$.     

A \textbf{weighted hypersurface} $X_d \subset \bP(a_0, \dots, a_{n+1})$ is the vanishing locus of a polynomial equation $f(x_0, \dots, x_n) \in \bC[x_0, \dots, x_n]$ that is weighted homogeneous of degree $d$.  A \textbf{weighted complete intersection} $X_{r_1, \dots,r_k} \subset \bP(a_0, \dots, a_{n+k})$ is the vanishing locus of a regular sequence of weighted homogeneous equations $(f_1,\dots, f_k)$ of degree $\deg f_i = r_i$.  Adjunction holds for weighted projective space in this setting, and the canonical divisor of $X_{r_1, \dots, r_k}$ has degree $-(a_0 + \dots + a_{n+k}) + r_1 + \dots + r_k$, so $X_{r_1, \dots, r_k}$ is Fano if and only if $r_1 + \dots + r_k \le a_0 + \dots + a_{n+k} - 1$. 

Let $I$ be the homogeneous ideal generated by the weighted polynomial equation $f(x_0, \dots, x_n) \in \bC[x_0, \dots, x_n]$. For a weighted hypersurface $X_d = \Proj \ \mathbb{C}[x_0, \dots, x_n]/ I$ let $C_{X_d}$ be the affine scheme $C_{X_d} \coloneqq \Spec \mathbb{C}[x_0, \dots, x_n]/ I$. The weighted hypersurface is said to be \textbf{quasismooth} if $C_{X_d}^* \coloneqq C_{X_d} \setminus \{ 0 \} $ is smooth. 
Analogously, a weighted complete intersection whose homogeneous ideal $J$ is defined by regular sequence of weighted homogeneous equations $(f_1,\dots, f_k)$ is quasismooth if $C_{X_{r_1, \dots, r_k}}^* \coloneqq (\Spec \mathbb{C}[x_0, \dots, x_{n+k}]/ J) \setminus \{ 0 \} $ is smooth.  The singularities of a quasismooth complete intersection arise only from the $\bC^*$-action on $\Spec \mathbb{C}[x_0, \dots, x_{n+k}] \setminus \{ 0 \}$, which can occur only at the intersections of $X_{r_1, \dots, r_k}$ with the singular locus of the weighted projective space $\bP(a_0, \dots, a_{n+k})$. 


In this paper, we will consider complete intersections $X_{r_1, \dots, r_k} \subset \bP(1^{n+1}, a_{n+1}, \dots, a_{n+k})$ such that $a_{n+i} \mid r_i$ for each $i$.  An advantage of these complete intersections is that the generic one is \textit{smooth}: 

\begin{proposition}\label{prop:smoothness-of-wted-ci}
    Let $X = X_{r_1, \dots, r_k} \subset \bP(1^{n+1}, a_{n+1}, \dots, a_{n+k})$ be a generic weighted complete intersection such that $a_{n+i} \mid r_i$.  Then, $X$ is smooth.
\end{proposition}

\begin{proof}
    Because smoothness is an open condition in families, it suffices to produce one smooth complete intersection of this type.  For ease of notation, denote by $[x_0: \dots : x_{n+k}]$ the coordinates on the weighted projective space and $b_{i} = \frac{r_i}{a_{n+i}}$.  Choose general homogeneous functions $f_1, \dots, f_k \in \bC[x_0, \dots ,x_n]$ of degree $ \deg f_i = r_i$ and define 
    	\begin{equation*}
		X\colon \begin{cases}
			f_1 + x_{n+1}^{b_1} =0\\
			f_2 + x_{n+2}^{b_2} = 0\\
			\quad \vdots \\
			f_{k-1} + x_{n+k-1}^{b_{k-1}} = 0 \\
			f_k+x_{n+k}^{b_k}=0.\\
		\end{cases} 
	\end{equation*}
    It is straightforward to check that $X$ is quasismooth from the Jacobian of these equations.  Furthermore, the intersection of $X$ with the locus $(x_0 = \dots = x_{n} = 0)$, which contains the singular locus of the weighted projective space, is empty, so $X$ is smooth.
\end{proof}

Because such a complete intersection is smooth, we may apply the following result to conclude its automorphism group is finite in most cases.

\begin{theorem}\cite[Theorem 1.3]{autsmthWCI}
    Let $X$ be a smooth well formed weighted complete intersection of dimension $n$. Suppose that either $n \geq 3$, or $K_X \ne 0$. Then the
group $\Aut(X)$ is finite unless $X$ is isomorphic either to $\bP^n$ or to a quadric hypersurface. 
\end{theorem}

\section{Some results on GIT}

In this section, we introduce the basic definitions in Geometric Invariant Theory (see \cite{mumford1994geometric}) and include some results on GIT stability of particular varieties that will be used in what follows. 

\begin{definition}
    Let $G$ be a reductive group acting by $G \to \SL_{n+1}$ on $\bP^n$, and let $\calL =\calO(1)$ be an ample line bundle on $\bP^n$.  A point $x \in \bP^n$ is 
    \begin{enumerate}
        \item \textbf{GIT-semistable} if there exists a $G$-invariant function $f \in H^0(\bP^n, \calO(d))^G$ such that $f(x) \ne 0$;  
        \item \textbf{GIT-polystable} if it is GIT-semistable and $G \cdot x$ is closed in the semistable locus; and 
        \item \textbf{GIT-stable} if it is GIT-polystable and the stabilizer subgroup $G_x$ is finite. 
    \end{enumerate}
\end{definition}

We also recall the Hilbert-Mumford criterion to detect GIT stability.  Let $G$ be a reductive group acting by $G \to \SL_{n+1}$ on $\bP^n =  \Proj \ \bC[x_0, \dots, x_n]$ and let $\lambda \colon \mathbb C^*\rightarrow G$. In suitable coordinates we can assume that the action is diagonal, that is, that it acts on the basis coordinates as $\lambda(t)\cdot x_i=t^{w_i}x_i$ for some integers $w_i$ such that $\sum_{i=0}^n w_i =0$. This action induces a natural grading on $R=\oplus_{w\in \mathbb Z} R_w$, where the piece with weight $w$, denoted by $R_w$, corresponds to all monomials $x^{\alpha}=x_0^{\alpha_0}\cdots x_n^{\alpha_n}$ such that  their weight $w_{\lambda}(x^{\alpha})=\sum_{i=0}^n \alpha_i w_i$ is exactly $w$. Recall that the Hilbert-Mumford function is defined as 
\begin{equation} \label{eq:HM}
    \mu(x,\lambda)=-\mathrm{min}_i\{w_i\, | \, x_i\not =0\} 
\end{equation}

A point $x \in R$ corresponds to a point $[x] \in \bP^n$, and by the Hilbert-Mumford criterion (\cite[Theorem 2.1]{mumford1994geometric}), with respect to the action of $G$ on $\bP^n$, this point is 

\begin{enumerate}
    \item \textbf{GIT-semistable} if $\mu(x, \lambda) \ge 0$ for all 1-PS subgroups $\lambda$; 
    \item \textbf{GIT-polystable} if $\mu(x, \lambda) > 0$ for all non-trivial 1-PS subgroups $\lambda$; and
    \item \textbf{GIT-stable} if $\mu(x, \lambda) > 0$ for all non-trivial 1-PS subgroups $\lambda$ and the stabilizer of $[x]$ under the action of $G$ on $\bP^n$ is finite.
\end{enumerate}

We list several well-known results in GIT.

\begin{lemma}\label{lem:GITstability-hypersurfaces}
    Let $X \subset \bP^n$ be a smooth hypersurface of degree $d \ge 2$.  If $d = 2$, then $X$ is GIT polystable.  If $d \ge 3$, then $X$ is GIT stable.
\end{lemma}

\begin{proof}
    For $d = 2$, this is Example 10.1 in \cite{dolgachev2003lectures}.  For $d \ge 3$, the polystability is \cite[Chapter 4, Proposition 4.2]{mumford1994geometric}.  Note that in \cite{mumford1994geometric}, Mumford uses the terminology \textit{stable} in place of \textit{polystable}, however in the course of the proof, shows that the stabilizer of a degree $d \ge 3$ hypersurface is finite.  
    
    The key observation in both proofs is that the discriminant hypersurface $\Delta \subset \bP(H^0(\bP^n, \calO(d)))$ is an $\Aut(\bP^n)$-invariant hypersurface that does not vanish on $[X]$, by the assumption that $X$ is smooth. 
\end{proof}

\begin{lemma}\label{lem:GITstability-finiteAut}
    Let $Y$ be a projective variety and suppose $G = \Aut(Y)$ is finite. Let $X \subset Y$ be a divisor $X \in |\calL|$, where $\calL$ is a very ample line bundle on $Y$.  Then, $[X]$ is a GIT stable point of $\bP(H^0(Y, \calL))$ with respect to the natural action of $G$.
\end{lemma}

\begin{proof}
    Let $P = \bP(H^0(Y, \calL))$ with induced action of $G$.  Given any line bundle $\calM$ on $P$, some power of $\calM$ admits a $G$-linearization (see, e.g. \cite[Corollary 7.2]{dolgachev2003lectures}), so we may assume $G$ acts linearly on $(P, \calO(1))$.  Because $G$ is finite, by \cite[Exercise 8.5]{dolgachev2003lectures}, every point of $P$ is stable.  Alternatively, by the Hilbert-Mumford criterion, there are no non-trivial one parameter subgroups of a finite group $G$, so $[X]$ is GIT stable. 
\end{proof}

The next result is standard but we include a proof for convenience of the reader.

\begin{lemma}\label{lem:GIT-of-(2,d)ci}
    Let $Q \subset \bP^n$ be a smooth quadric hypersurface and $X \in |\calO_{Q}(d)|$ a smooth divisor for $d \ge 2$.  Then, $[X]$ is a GIT stable point of $\bP(H^0(Q, \calO_Q(d)))$ with respect to the natural action of $\Aut(Q) = SO(n+1)$. 
\end{lemma}

\begin{proof}
    Because $L = \calO_{Q}(d)$ is very ample, the generic element $X \in |L|$ is smooth by Bertini's Theorem.  Then, the discriminant is given by the dual variety $\Delta = Q^{\vee}$ in $P = \bP(H^0(Q, L))$, which an $SO(n+1)$-invariant divisor that is $0$ at a point $[X], X \in |L|$ if and only if $X$ is singular (see, for example, \cite[Chapter I]{GKZbook}).  We conclude that a smooth complete intersection $[X] \in P$ is GIT semistable. 
    
    To prove that it is polystable, because $\Delta$ is an $SO(n+1)$-invariant function, it is constant on orbit closures.  Therefore, the closure of the orbit of a smooth point $[X]$ must be contained in the complement of $\Delta$.  Suppose $SO(n+1) \cdot [X]$ is not closed.  Then, there exists a smooth complete intersection $Y$ such that $[Y] \in \overline{SO(n+1)\cdot [X]} - SO(n+1) \cdot X$.  By the Orbit-Stabilizer theorem, we must have the dimension of the stabilizer of $Y$ larger than that of $X$.  However, because any smooth complete intersection $Y$ has finite automorphism group by \cite{LW84} or \cite[Theorem 3.1]{benoist2013separation}, this is a contradiction and therefore $SO(n+1) \cdot [X]$ is closed.  We therefore conclude that $[X]$ is GIT polystable, and as it has finite automorphism group by \cite{LW84} or \cite[Theorem 3.1]{benoist2013separation}, it is in fact GIT stable. 
\end{proof}

We will also use the GIT stability of simple normal crossing sums of quadrics. 

\begin{lemma}\label{lem:sncquadrics}
    Let $G=\mathrm{SL}(n+1)$ and consider its natural action on the space of degree $d$ polynomials in $n+1$ variables, that is, $V=\mathrm{Sym}^{d}\big(\mathbb C{^{n+1}}\big)^*$. Let $B = \sum B_l$ be the sum of $k$ smooth quadrics such that $B$ is simple normal crossing in $\mathbb P^n$. Then, $[B]$ is a GIT-polystable point of $\bP(\mathrm{Sym}^{2k}\big(\mathbb C{^{n+1}}\big)^*)$. Moreover, if $k\geq 2$ then $[B]$ is GIT-stable. 
\end{lemma}

\begin{proof}
    Let $f_1,\ldots,f_k \in \mathbb C[x_0,\ldots,x_n]$ be degree $2$ homogeneous polynomials defining $k$ smooth quadrics $B_1,\ldots, B_k$ such that $B=\sum_{l=1}^k B_l$ is simple normal crossing. Let $\lambda \colon \mathbb C^*\rightarrow G$ be a one parameter subgroup of $G$ acting diagonally i.e., such that $\lambda(t)\cdot x_j = t^{\lambda_j}x_j$ with $\lambda_j \in\mathbb Z$ and $\sum_{i=0}^n \lambda_i =0$. We adopt the convention that for $F\in V$, we have $(\lambda(t)\cdot F)(x)=F(\lambda(t)^{-1}x)$ and without loss of generality, we order the the weights as $\lambda_0\geq \lambda_1 \geq \cdots \geq \lambda_{n}$.
    
    Since smooth quadrics are GIT-polystable (see \cite[Example~10.1]{dolgachev2003lectures}) it follows from the Hilbert-Mumford criterion that $\mu(f_l,\lambda)\geq 0$ for all $1\leq l\leq k$ and by additivity
     
\begin{equation} \label{eq:AddHM}
\mu\bigg(\prod_{l=1}^k f_l,\lambda \bigg)=\sum_{l=1}^k \mu \big( f_l,\lambda \big)\geq 0.   
\end{equation}
Assume $k\geq 2$ and let $F=\prod_{l=1}^k f_l$. If $\mu(F,\lambda)>0$, then the conclusion follows so we assume $\mu(F,\lambda)=0$ for all $\lambda$. In that that case we have $\mu(f_l,\lambda)=0$ for all $l$, or, equivalently for smooth quadrics, $\mu(f_l,\lambda)\leq 0$. Write $B_l=x^{T}A_lx$ for a symmetric $(n+1) \times (n+1)$ matrix $A_l$ for each $1\leq l\leq k$. Since 
$$
\mu(f_l,\lambda)=\max\{\lambda_i+\lambda_j \,\, \colon \,\, \text{the coefficient of $x_i x_j$ in $F$ is non-zero}\}
$$ 
the condition $\mu(f_l,\lambda)\leq 0$ is equivalent to 
\begin{equation} 
\label{eq:uppantidiagon}
((a_l)_{ij})=0 \quad \text{whenever} \,\, \lambda_i+\lambda_j>0    
\end{equation}
that is, the vanishing of the entry in $A_l$ corresponding to row $i$ and column $j$, for every $1\leq l\leq k$.  Hence, to show that $\mu(F,\lambda)>0$ it suffices to show that no $\lambda$ can make all $A_l$ satisfy condition \eqref{eq:uppantidiagon}. We assume, by contradiction, that there exists a non-trivial $\lambda$ such that for all $1\leq l\leq k$ the symmetric matrix $A_l$ satisfies condition \eqref{eq:uppantidiagon}. We will show that, in this case, any two quadrics will share a tangent direction, contradicting the simple normal crossing assumption.

For each $1\leq l\leq k$ we have 
$$
\det A_l = \sum_{\sigma \in S_{n+1}}\mathrm{sgn}(\sigma)(a_l)_{1,\sigma(1)}\cdots (a_l)_{n+1, \sigma(n+1)}.
$$
Since $A_l$ is invertible, there exists some $\sigma \in S_{n+1}$ such that we have $((a_l)_{i,\sigma(i)})\not =0 $ for all $i$, which is possible only if $\lambda_i+\lambda_{\sigma(i)} \leq 0$ by condition \eqref{eq:uppantidiagon}. But by definition of $\lambda$, we have $0=\sum \lambda_i+\sum \lambda_{\sigma (i)}=\sum (\lambda_i + \lambda_{\sigma(i)})$ which implies that $\lambda_i + \lambda_{\sigma{(i)}}=0$ for all $i$. Hence $\lambda$ is symmetric around $0$ and we may write $\lambda_n = -\lambda_0$, $\lambda_{n-1} = -\lambda_1$, etc. 

Consider the grading of $W = \bC^{n+1}$ induced by $\lambda$ given by $W=W_{\lambda_{0}} \oplus \bigoplus_{|\lambda|<\lambda_0} W_{\lambda}\oplus W_{-\lambda_{0}}$ where $W_{w}=\langle e_j \,\,\colon \,\, \lambda_j=w \rangle$ and $e_i$ for $0\leq i\leq n$ forms a basis for $W$. Let $d=\dim W_{\lambda_0}=\dim W_{-\lambda_0} \geq 1$. From the previous paragraph we can deduce that $A_l$ is of the form
$$
A_l=\begin{pmatrix}
  0 & 0 & M_l^T \\
  0 & * & * \\
  M_l & * & * 
\end{pmatrix}
$$
where the top left $0$ represents the $d\times d$ zero matrix and $M_l$ is some $d\times d$ matrix. Consider the linear space $L:=\bP(W_{\lambda_0})$. Then it is clear from the shape of $A_l$ that
\begin{enumerate}
    \item 
each quadric contains $L$;
\item for each point $x \in L$ (viewed as the point $x = \begin{pmatrix} x \\ 0 \end{pmatrix} \in W$ where the second entry $0$ represents the $0$ vector of length $n+1 - d$), we have $A_lx=\begin{pmatrix}  0 \\M_lx\end{pmatrix}$ where the first $0$ represents the $0$ vector of length $n+1 - d$, and
\item $M_l$ is invertible.
\end{enumerate}
This holds for each $1 \le l \le k$ by construction.  For simplicity, we will write condition (2) as $A_l x = M_l x$.  

Since the gradient of a quadric $B_l$ is given by $2A_l x$, if we find a common point $x$ of two quadrics, say $B_1$ and $B_2$, such that $A_1x=cA_2x$ for some non-zero complex number $c$, then we have shown $B_1$ and $B_2$ are tangent at $x$.  To see this, let $x\in L$ be a point in $\ker(M_1-tM_2)$. The determinant $\det (M_1-tM_2)$ is a polynomial of degree $d\geq 1$ in $t$ whose $t^d$ coefficient is $(-1)^d\det M_2$ and constant coefficient $\det M_1$, both non-zero. Hence $\det(M_1-tM_2)= 0$  has a non-zero solution $c$. Therefore we can choose $x$ non-zero. We conclude that
$$
A_1x = M_1x=cM_2x=cA_2x
$$
contradicting the fact that $B$ is simple normal crossing. Hence $\mu(F,\lambda)>0$ and by the Hilbert-Mumford criterion $[B]$ is GIT-stable.
\end{proof}


\section{K-stability of Fano varieties}

In this section, we include some background on K-stability and results that will be used in what follows.

\subsection{K-stability}

Most of the material in this section comes from \cite{Chenyang}. We work over the field of complex numbers. A pair $(X,\Delta)$ consists of a normal irreducible variety $X$ and a $\mathbb Q$-divisor $\Delta$ such that $K_X+\Delta$ is $\mathbb Q$-Cartier.  The pair $(X,\Delta)$ is projective if $X$ is projective.  A pair $(X,\Delta)$ is log Fano if $X$ is projective, $(X,\Delta)$ is klt and $-(K_X+\Delta)$ is ample. 


\begin{definition}[\cite{Tia97, Don02}]
Let $(X,\Delta=\sum_{i=1}^k b_i \Delta_i)$ be a pair and let $L$ be an ample line bundle on $X$.
\begin{enumerate}[label=(\alph*)]
\item A \emph{test configuration} of index $r$ $(\cX;\cL_r)/\bA^1$ of $(X;L)$ consists of the following data:
\begin{itemize}
 \item a variety $\cX$ together with a flat projective morphism $\pi:\cX\to \bA^1$;
 \item a $\pi$-ample line bundle $\cL_r$ on $\cX$;
 \item a $\bG_m$-action on $(\cX;\cL_r)$ such that $\pi$ is $\bG_m$-equivariant with respect to the standard action of $\bG_m$ on $\bA^1$ via multiplication;
 \item for $t \ne 0$, an isomorphism $\phi_t:\cX_t \cong X$ such that $\phi_t^*L^r = \calL_r$.
\end{itemize}
 
\item A \emph{test configuration} $(\cX,\mathcal{D};\cL_r)/\bA^1$ of $(X,\Delta;L)$ consists of the following data:
\begin{itemize}
\item a test configuration $(\cX;\cL_r)/\bA^1$ of $(X;L)$;
\item a formal sum $\mathcal{D}=\sum_{i=1}^k c_i \mathcal{D}_i$ of codimension one closed integral subschemes $\mathcal{D}_i$ of $\cX$ such that $\mathcal{D}_i$ is the Zariski closure of $\Delta_i\times(\bA^1\setminus\{0\})$ under the identification between $\cX\setminus\cX_0$ and $X\times(\bA^1\setminus\{0\})$.
\end{itemize}
\item A test configuration $(\cX,\mathcal{D};\cL_r)/\bA^1$ is called a \emph{normal} test configuration if $\cX$ is normal. 
A normal test configuration is called a \emph{product} test configuration if \[
(\cX,\mathcal{D};\cL_r)\cong(X\times\bA^1,\Delta\times\bA^1;pr_1^* L^r\otimes\cO_{\cX}(k\cX_0))
\] for some $k\in\bZ$. A product test configuration is called a \emph{trivial} test configuration if the above isomorphism is $\bG_m$-equivariant with respect to the trivial $\bG_m$-action on $X$ and the standard $\bG_m$-action on $\bA^1$. 
\item 
Let $(X,\Delta)$ be a log Fano pair. Let $L$ be an ample $\bQ$ line bundle on $X$ such that $L\sim_{\bQ}-(K_X+\Delta)$. A normal test configuration $(\cX,\mathcal{D};\cL_r)/\bA^1$ is called a \emph{special test configuration} if $\cL_r\sim_{\bQ}-r(K_{\cX/\bA^1}+\mathcal{D})$ and $(\cX,\mathcal{D}+\cX_0)$ is plt. In this case, we say that $(X,\Delta)$ \emph{specially degenerates to} $(\cX_0,\mathcal{D}_0)$ which is necessarily a log Fano pair.
\item Given a test configuration $(\cX,\mathcal{D};\cL_r)/\bA^1$ we may glue this to the trivial test configuration $(X,\Delta; L^r) \times \bA^1 \to \bA^1 $ along $\bA^{1} \setminus \{ 0\}$ to produce the \textit{$\infty$-trivial} compactification $\overline{\pi}: (\overline{\cX}, \overline{\mathcal{D}}; \overline{\cL}_r) \to\bP^1$.
\end{enumerate}
\end{definition}

Associated to any test configuration, we can define a \textit{Futaki invariant}, coming from the weight of the $\bG_m$-action.  We provide a definition only in the case of special test configurations due to \cite{LX14}.  

\begin{definition}
    For any log Fano pair $(X,\Delta)$ and special test configuration $(\calX, \mathcal{D}; \calL_r)$, the \textit{Futaki invariant} is 
    \[ \Fut(\calX, \mathcal{D}; \calL_r) = - \frac{1}{2(-K_X-\Delta)^n(n+1)} \left(-K_{\overline{\calX}/\bP^1} - \overline{\mathcal{D}}\right)^{n+1} ,\]
    where $(\overline{\calX},\overline{\mathcal{D}})$ is the $\infty$-trivial compactification.
\end{definition}

\begin{definition}
     Let $(X,\Delta)$ be a log Fano pair.  Then, $(X,\Delta)$ is said to be:
\begin{enumerate}[label=(\roman*)]
    \item \emph{K-semistable} if $\Fut(\cX,\mathcal{D};\cL_r)\geq 0$ for any special test configuration of any index $r$;
    
    \item  \emph{K-stable} if it is K-semistable and $\Fut(\cX,\mathcal{D};\cL_r)=0$ for a special test configuration $(\cX,\mathcal{D};\cL_r)/\bA^1$ if and only if it is a trivial test configuration; and
 
\item \emph{K-polystable} if it is K-semistable and $\Fut(\cX,\mathcal{D};\cL_r)=0$ for a special test configuration $(\cX,\mathcal{D};\cL_r)/\bA^1$ if and only if it is a product test configuration.
\end{enumerate}
\end{definition}

 We repeatedly use the fact that if $(X,\Delta)$ is a log Fano pair such that  $\mathrm{Aut}(X,\Delta)$ is finite, then $(X,\Delta)$ is K-polystable if and only if it is K-stable, which follows from \cite{blum2019uniqueness} (see \cite[Corollary 2.2.5]{cheltsov2026kstabilityfanosfanos}).

The K-stability of a Fano variety or log Fano pair places constraints on its singularities: 

\begin{theorem}\cite{Odaka}
    Suppose $(X,\Delta)$ is a K-semistable log Fano pair.  Then, $(X,\Delta)$ has klt singularities. 
\end{theorem} 

Using the general formula for the Futaki invariant (cf. \cite{Chenyang}), one can define K-stability for an arbitrary polarized variety.  In the case that $(X, \Delta)$ is a log Calabi-Yau pair, i.e. $K_X + \Delta \sim_{\bQ} 0$, this is equivalent to having klt or log canonical singularities: 

\begin{theorem}\cite{Odaka}\label{odaka-CY}
    A log Calabi-Yau pair $(X,\Delta)$ is K-stable (respectively, K-semistable) if and only if $(X,\Delta)$ is klt (respectively, log canonical). 
\end{theorem} 

For the definition of klt and lc singularities, we refer the reader to \cite{KM,Kol13}.

\subsection{K-stability of quotients and cyclic covers}

K-stability behaves well with respect to finite covers: 

\begin{theorem}\cite{liu2022equivariant, ZhuangEquivariant}\label{thm:finitecovers}
    Let $X$ be a Fano variety and $G$ a finite group acting on $X$.  Let $Z = X/G$ and $\pi: X \to Z$ the quotient map.  Define $Y \subset Z$ as the divisor such that the following equality holds: 
    \[ K_X = \pi^*(K_Z + Y).\]
    Then, $X$ is K-semistable (respectively, polystable) if and only if $(Z,Y)$ is K-semistable (respectively, polystable). 
\end{theorem}

We will also use the following Lemma to relate K-stability of an $n-1$-dimensional Fano variety to K-stability of log Fano pairs of dimension $n$, which is a special case of \cite[Proposition 2.11]{LiuZhuangSharpness}. 

\begin{definition}
    Let $Y$ be a projective variety and $\calL$ an ample $\bQ$-line bundle on $Y$.  The \textbf{projective orbifold cone} over $Y$ with respect to $\calL$ is 
    \[ C_p(Y, \calL) = \Proj \big( \sum_{m \ge 0} \sum_{i+j = m} H^0(Y,\calL^{[i]})\cdot z^j\big).\]
\end{definition}

\begin{lemma}\label{lem:Liu-Zhuang-cone}
    Let $Z$ be a Gorenstein Fano variety of dimension $n$ and let $Y \subset Z$ be a Fano hypersurface $Y \in |\calO_Z(d)|$ for some $d \ge 2$, where $\calO_Z(1)$ is a very ample line bundle on $Z$ such that $\calO(-K_Z) \cong \calO_Z(k)$ for some $0<k \le d(n+1)$.  
    Let $C_Y = C_p\big(Y, \mathcal{O}(d)\big)$ be the projective orbifold cone over $Y$ with respect to the line bundle $\mathcal{O}(d)$. Then, $Y$ is K-polystable if and only if $\big(C_Y, (1 - \frac{k-d}{nd})Y\big)$ is K-polystable, where $Y \subset C_Y$ is the hyperplane section at infinity. 
\end{lemma}

\begin{proof}
    We follow the notation in \cite[Proposition 2.11]{LiuZhuangSharpness}.  Letting $M = \mathcal{O}_Z(d)|_Y$, then by adjunction $\deg M = \deg (-\frac{d}{k-d} K_Y)$.  Note that the hypothesis that $Y$ is Fano implies $d < k$.  Let $r = \frac{k-d}{d} \le n$. Notice that $r>0$ since $Y$ is Fano and $r \le n$ by the assumption that $k \le d(n+1)$. Then, by \cite[Proposition 2.11]{LiuZhuangSharpness}, $Y$ is K-polystable if and only if $\Big(C_Y, \big(1 - \frac{r}{n}\big)Y\Big)$ is K-polystable, which is precisely the pair $\Big(C_Y, (1 - \frac{k-d}{nd})Y\Big)$.
\end{proof}

For $(Z,Y)$ as in the previous lemma, the K-polystability of $Y$ implies that $\big(C_Y, (1 - \frac{k-d}{nd})Y\big)$ is K-polystable.  By \cite[Theorem 1.3]{Odaka}, the K-polystability of $Y$ also implies that $Y$ is klt and hence $(Z,Y)$ is a plt pair by \cite[Proposition 5.50]{KM}.  Then, by \cite[Lemma 2.12]{LiuZhuangSharpness}, the pair $(Z,Y)$ specially degenerates to $(C_Y,Y)$: there exists an isotrivial family over $\bA^1$ with generic fiber over $t \ne 0$ $(Z,Y)$ and special fiber over $t = 0$ $(C_Y,Y)$.  As K-semistability is an open condition in families by \cite{blum2022openness}, the K-(semi/poly)stability of $(C_Y, \alpha Y)$ for any $\alpha$ will therefore imply the K-semistability of the pair $(Z, \alpha Y)$.  Combined with Lemma \ref{lem:Liu-Zhuang-cone}, we conclude the following: 

\begin{proposition}\label{prop:KstabofYimpliesKstabofpair}
    Let $Z$ be a Gorenstein Fano variety of dimension $n$ and let $Y \subset Z$ be a Fano hypersurface $Y \in |\calO_Z(d)|$ for some $d \ge 2$, where $\calO_Z(1)$ is a very ample line bundle on $Z$ such that $\calO(-K_Z) \cong \calO_Z(k)$ for some $0<k \le d(n+1)$.  If $Y$ is K-polystable, then the pair $(Z,(1 - \frac{k-d}{nd})Y)$ is K-semistable.
\end{proposition}

K-stability is also defined for log Calabi Yau pairs so we can extend the previous result to the case that $Y\subset Z$ is not necessarily Fano. By Theorem \ref{odaka-CY}, we can generalize the previous proposition to this case as follows:

\begin{proposition}\label{prop:KYnefimpliesKstabofpair}
    Let $Z$ be a Gorenstein Fano variety of dimension $n$ and let $Y \subset Z$ be a hypersurface $Y \in |\calO_Z(d)|$ for some $d \ge 2$ with $K_Y$ nef, where $\calO_Z(1)$ is a very ample line bundle on $Z$ such that $\calO(-K_Z) \cong \calO_Z(k)$ for some $d\geq k$.  If $(Z,\frac{k}{d}Y)$ is klt (respectively, log canonical), then the pair $(Z,\frac{k}{d}Y)$ is K-stable (respectively, K-semistable).
\end{proposition}

Next, we have a useful result known as \textit{interpolation} \cite[Proposition 2.13]{ascher2024wall}: 

\begin{proposition}\label{prop:k-interpolation}
Let $X$ be a $\bQ$-Fano variety. Let $D$ and $\Delta$ be effective $\bQ$-divisors on $X$ satisfying the following properties:
\begin{itemize}
    \item Both $D$ and $\Delta$ are rational multiples of $-K_X$.
    \item $-K_X-D$ is ample, and $-K_X-\Delta$ is nef.
    \item The log pairs $(X,D)$ and $(X,\Delta)$ are K-(poly/semi)stable and K-semistable, respectively.
\end{itemize}
Then we have
\begin{enumerate}
    \item If $D\neq 0$, then $(X,tD+(1-t)\Delta)$ is K-(poly/semi)stable for any $t\in (0,1]$.
    \item If $D=0$, then $(X,(1-t)\Delta)$ is K-semistable
    for any $t\in (0,1]$.
    \item If $\Delta\sim_{\bQ}-K_X$ and $(X,\Delta)$ is klt, then $(X,tD+(1-t)\Delta)$ is uniformly K-stable for any $t\in (0,1)$.
\end{enumerate}
\end{proposition}

From interpolation, we can conclude the following. Roughly, this says if $Y \subset Z$ are as in the previous propositions, the K-polystability of $Z$ and $Y$ will imply the K-stability of the pair $(Z,\alpha Y)$ for large ranges of $\alpha$.  

\begin{lemma}\label{lem:semistabilityofpair}
    Let $Z$ be a K-polystable Gorenstein Fano variety of dimension $n$ and let $Y \subset Z$ be a hypersurface $Y \in |\calO_Z(d)|$ for some $d \ge 2$, where $\calO_Z(1)$ is a very ample line bundle on $Z$ such that $\calO(-K_Z) \cong \calO_Z(k)$ for some $k >0$.
    
    \begin{enumerate}
        \item Suppose $k \le d(n+1)$.  If $Y$ is Fano and K-polystable, then $(Z, cY)$ is K-semistable for all $c \in (0, 1 - \frac{k-d}{nd})$.
        \item If the log Calabi Yau pair 
        $(Z,\frac{k}{d}Y)$ is klt, then $(Z,cY)$ is K-stable for all $c \in (0, \frac{k}{d})$.
        \item If the log Calabi Yau pair 
        $(Z,\frac{k}{d}Y)$ is log canonical, then $(Z,cY)$ is K-semistable for all $c \in (0, \frac{k}{d})$. 
    \end{enumerate}
\end{lemma}

\begin{proof}
    Assume first that $Y$ is Fano. Since $Y$ is K-polystable, we conclude $(Z, (1- \frac{k-d}{nd})Y)$ is K-semistable by Proposition \ref{prop:KstabofYimpliesKstabofpair}.  By assumption, $Z$ is K-polystable, so by interpolation applied to the pairs $(Z,0)$ and $(Z,(1- \frac{k-d}{nd})Y)$, $(Z, c Y)$ is K-semistable for all $c \in (0, 1 - \frac{k-d}{nd})$.  

     Now, assume that the log Calabi Yau pair $(Z,\frac{k}{d}Y)$ is klt (respectively, log canonical).  In this case, necessarily $k \le d$, so $K_Y$ is nef by adjunction.  By Proposition \ref{prop:KYnefimpliesKstabofpair}, $(Z, \frac{k}{d}Y)$ is K-stable (respectively, K-semistable).  Therefore, by interpolation, we conclude $(Z, c Y)$ is K-stable (respectively, K-semistable) for all $c \in (0,\frac{k}{d})$. 
\end{proof}

Finally, we prove that the previous lemma can be upgraded to conclude K-polystability of $(Z,cY)$ when $Y$ is K-polystable or, if $K_Y$ is nef, GIT-polystable.  First, we have a lemma well known to experts (see, e.g. \cite[Theorem 3.4]{OSS}).

\begin{lemma}\label{lem:KimpliesGIT}
    Let $Z$ be a projective variety and $\calL$ a very ample line bundle on $Z$.  Let $P = \bP(H^0(Z, \calL))$ be the space of sections of $\calL$ and $G$ a reductive group acting on $Z$ with linearized action on $P$.  Denote by $[P/G]$ the GIT quotient stack.  Let $Y \in |\calL|$ and suppose $Y$ is a K-polystable Fano variety.  Then, $[Y]$ is a GIT polystable point of $[P/G]$. 
\end{lemma}

\begin{proof}
    This holds for finite groups $G$ by Lemma \ref{lem:GITstability-finiteAut} so assume $G$ is infinite and let $\lambda$ be a 1-parameter subgroup in $G$.  This follows from \cite{OSS} so we only provide a sketch.  By \cite{PaulTianCMI}, the Hilbert-Mumford weight $\mu_{\lambda}(Y)$ is a positive scalar multiple of the Futaki invariant of the associated test configuration.  Therefore, because $Y$ is K-polystable, the Futaki invariant is nonnegative so the Hilbert-Mumford weight is non-negative.  This implies that $[Y]$ is GIT semistable.  However, it also has closed orbit: as it was K-polystable, the only test configurations with Futaki invariant $0$ are of product type and hence have central fiber isomorphic to $Y$, and therefore the limit of any one-parameter subgroup with trivial Hilbert-Mumford weight must be isomorphic to $Y$.
\end{proof}

Using the previous lemma, we obtain the following. 

\begin{proposition}\label{prop:polystabilityofpair}
    Let $Z$ be a K-polystable Gorenstein Fano variety of dimension $n$ and let $Y \subset Z$ be a hypersurface $Y \in |\calO_Z(d)|$ for some $d \ge 2$, where $\calO_Z(1)$ is a very ample line bundle on $Z$ such that $\calO(-K_Z) \cong \calO_Z(k)$ for some $k >0$.  
    
    \begin{enumerate}
        \item Suppose $k \le d(n+1)$.  If $Y$ is Fano and K-polystable, then $(Z, cY)$ is K-polystable for all $c \in (0, 1 - \frac{k-d}{nd})$.
        \item If the log Calabi Yau pair 
        $(Z,\frac{k}{d}Y)$ is klt, then $(Z,cY)$ is K-stable for all $c \in (0, \frac{k}{d})$.
        \item If the log Calabi Yau pair 
        $(Z,\frac{k}{d}Y)$ is log canonical  and $Y$ is GIT polystable with respect to the action of $G = \Aut(Z)$, then $(Z,cY)$ is K-polystable for all $c \in (0, \frac{k}{d})$.
    \end{enumerate}
\end{proposition}

\begin{proof}
    Part (2) is \ref{lem:semistabilityofpair}. For parts (1) and (3), note by \cite{ABHLX}, the K-polystability of $Z$ implies $G = \Aut(Z)$ is reductive, so the GIT quotient is defined.  If $Y$ is Fano, by Lemma \ref{lem:KimpliesGIT}, then the K-polystability of $Y$ implies that $Y$ is GIT polystable, so in either case we may assume $Y$ is GIT polystable. 
    
    By Lemma \ref{lem:semistabilityofpair}, we have $(Z,cY)$ is K-semistable for $c$ as in the statement.  Because $(Z,\epsilon Y)$ is K-semistable for any $\epsilon \ll 1$, by \cite[Theorem 10.3]{LZnonproportional}, the GIT polystability of $Y$ implies that $(Z, \epsilon Y)$ is K-polystable.  Therefore, we can in fact apply interpolation to the pairs $(Z, \epsilon Y)$ and $(Z,(1- \frac{k-d}{nd})Y)$ to conclude $(Z, c Y)$ is K-polystable for all $c \in (0, 1 - \frac{k-d}{nd})$ (if $Y$ is Fano) or $c \in (0, \frac{k}{d})$ (if $K_Y$ is nef).
\end{proof}

The results in this section will be used repeatedly in the following way: 

\begin{theorem}\label{thm:kstab-of-cover}
    Let $X$ be a Fano variety that is a degree $m\geq 2$ cyclic cover of an $n$-dimensional Gorenstein Fano variety $Z$ branched over a hypersurface $Y \in |\calO_Z(d)|$ such that $m \mid d$ and $d \ge 2$, where $\calO_Z(1)$ is a very ample line bundle on $Z$ such that $\calO(-K_Z) \cong \calO_Z(k)$ for some $k > 0$.

    Assume that $Z$ is K-polystable. In addition,
    
    \begin{enumerate}
        \item If $Y$ is Fano, assume that $Y$ is K-polystable and $k\leq  n\frac{d}{m}+d-1$; or
        \item assume that $(Z,\frac{k}{d}Y)$ is log canonical and $Y$ is GIT polystable with respect to the action of $G = \Aut(Z)$.
    \end{enumerate}

    Then, $X$ is K-polystable.  If $\Aut(X)$ is finite, then $X$ is K-stable.
\end{theorem}

\begin{proof}
    By \cite{liu2022equivariant} (Theorem \ref{thm:finitecovers}), $X$ is K-polystable if and only if the pair $(Z, (1 - \frac{1}{m})Y)$ is K-polystable.  By Proposition \ref{prop:polystabilityofpair}, we know $(Z, cY)$ is K-polystable for any $c \in (0, 1 - \frac{k-d}{nd})$ if $Y$ is Fano and for any $c \in (0, \frac{k}{d})$ if $(Z, \frac{k}{d}Y)$ is log canonical.

    Suppose we are in case (1) and $k \leq n\frac{d}{m}+d-1$. Since $m$ divides $d$ the right-hand side is a positive integer. Then,
\[
k<n\frac{d}{m}+d \iff k-d< n\frac{d}{m} \iff \frac{k-d}{nd} < \frac{1}{m} \iff 1-\frac{k-d}{nd} > 1-\frac{1}{m}.
\] 
Notice that $k<n\frac{d}{m}+d <d(n+1)$ so Proposition \ref{prop:polystabilityofpair} applies and therefore $(Z, (1 - \frac{1}{m})Y)$ is K-polystable and the result follows.

    Now suppose we are in case (2).  Note that the hypothesis that $X$ is Fano implies $k > \frac{m-1}{m} d $ and therefore,
    \[ 1 - \frac{1}{m} = \frac{m-1}{m} < \frac{k}{d}.\]  Hence, $(Z, (1 - \frac{1}{m})Y)$ is K-polystable and the result follows.

    We therefore conclude $X$ is K-polystable, and hence if $\Aut(X)$ is finite, $X$ is K-stable. 
\end{proof}

\section{K-polystable branched covers}\label{sec:covers}

In this section, we prove the K-polystability of several Fano varieties arising as branched covers using the results in the previous section.  Most of these are known to experts or have appeared earlier in the literature.

\begin{corollary}
    Suppose $Y \subset \bP^n$ is a smooth hypersurface of degree $d \ge 2$.  Then, 
        \begin{enumerate}
        \item if $d \le n$ and $Y$ is K-polystable, then $(\bP^n, cY)$ is K-polystable for all $c \in \Big (0, 1 - \frac{n+1-d}{nd}\Big)$, and
        \item if $d \ge n+1$, $(\bP^n,cY)$ is K-polystable for all $c \in \big(0, \frac{n+1}{d}\big)$.
    \end{enumerate}
\end{corollary}

\begin{proof}
    First observe that $\bP^n$ is K-polystable. Because $Y$ is smooth, it is GIT polystable by Lemma \ref{lem:GITstability-hypersurfaces}.  The hypotheses of Proposition \ref{prop:polystabilityofpair} are hence all satisfied, and in the notation of Proposition \ref{prop:polystabilityofpair} with $Z = \bP^n$, we have $k = n+1$, so the result follows. 
\end{proof}

\begin{corollary}
    Suppose $Z$ is a smooth K-stable Fano variety and let $Y \subset Z$ be a smooth divisor $Y \in |\calO_Z(d)|$ for some $d \ge 2$, where $\calO_Z(1)$ is a very ample line bundle on $Z$ such that $\calO(-K_Z) \cong \calO_Z(k)$ for some $0<k \le d(n+1)$. Then, 
    \begin{enumerate}
        \item if $d < k$ and $Y$ is a K-polystable Fano variety, $(Z, cY)$ is K-stable for all $c \in (0, 1 - \frac{k-d}{nd})$, and
        \item if $d \ge k$, $(Z,cY)$ is K-stable for all $c \in (0, \frac{k}{d})$.
    \end{enumerate}
\end{corollary}

\begin{proof}
    Because $Z$ is K-stable, $\Aut(Z)$ is finite, and hence any $Y \in |\calO_Z(d)|$ is GIT stable by Lemma \ref{lem:GITstability-finiteAut}.  Because $Y$ and $Z$ are smooth, $(Z,Y)$ is log canonical.  Therefore, the hypotheses of Proposition \ref{prop:polystabilityofpair} are satisfied and the result follows.
\end{proof}

We prove several corollaries on the stability of hypersurfaces and finite covers, and in particular provide algebraic proofs (and strengthenings) of the results of Arezzo, Ghigi, and Pirola in \cite{AGP06}.  First, we reproduce the following result of Liu and Zhu for the convenience of the reader (\cite[Theorem 5.1]{liu2022equivariant}):

\begin{corollary}\label{cor:hypersurface-cover}
    Let $Y = (f(x_0, \dots, x_n) = 0) \subset \mathbb{P}^{n}$ be a K-polystable Fano hypersurface or a smooth Calabi-Yau hypersurface of degree $d \geq 2$.  If $X = (f(x_0, \dots, x_n) + x_{n+1}^d = 0) \subset \mathbb{P}^{n+1}$ is the degree $d$ cyclic cover of $Y$, then $X$ is K-polystable. If $d \geq 3$ and $X$ is smooth, then $X$ is K-stable.
\end{corollary}

\begin{proof}
    Consider the map $\mathbb{P}^{n+1} \to \mathbb{P}(1^{n},d) $ given by $[x_0: \dots:x_n:x_{n+1}] \mapsto [x_0: \dots :x_n:x_{n+1}^d]$.  The image of $X$ is $\mathbb{P}^n$ and the induced map $X \to \mathbb{P}^n$ is a $d:1$ cover branched over $Y$.  Because $Y$ is smooth in the Calabi Yau case, it is GIT polystable by Lemma \ref{lem:GITstability-hypersurfaces}, and therefore Theorem \ref{thm:kstab-of-cover} implies that $X$ is K-polystable.  If $d \ge 3$, by \cite{MM63} we conclude $\mathrm{Aut}(X)$ is finite because $X$ is a smooth Fano hypersurface.  This implies $X$ is K-stable. 
\end{proof}

Our next corollary is a generalization of the stability of the Fermat hypersurface and \cite[Case (a), pg. 3]{AGP06}:

\begin{corollary}\label{cor:fermat-generalization}
    Let $X = (f(x_0, \dots, x_{n-k+1}) + x_{n-k+2}^d + \dots + x_{n+1}^d = 0) \subset \mathbb{P}^{n+1}$ be a Fano hypersurface be such that $f(x_0, \dots, x_{n-k+1})$ defines a smooth hypersurface of degree $d$ in $\mathbb{P}^{n-k+1}$ and $k \ge n+2 -d$.  Then, $X$ is K-polystable.  If $d  \geq 3$, then $X$ is K-stable. 
\end{corollary}

\begin{proof}
    It suffices to prove the K-polystability because, if $d > 2$, then $\mathrm{Aut} (X)$ is finite by the same argument as in the previous proof, so the stability follows.  

    Because $X$ is Fano, $d < n+2$.  Let $m$ be such that $d = n+2-m$.  Because $k\ge n+2 - d$, we have $k \ge m$, so we may write $X = (f(x_0, \dots, x_{n-k+1}) + x_{n-k+2}^d + \dots + x_{n-m+1}^d + x_{n-m+2}^d  + \dots + x_{n+1}^d = 0)$.  By induction and Theorem \ref{thm:kstab-of-cover}, it suffices to show that the Fano hypersurface $Z = (f(x_0, \dots, x_{n-k+1}) + x_{n-k+2}^d + \dots + x_{n-m+1}^d + x_{n-m+2}^d=0 ) \subset \mathbb{P}^{n+2-m}$ is K-polystable.  This follows directly from Theorem \ref{thm:kstab-of-cover} because the Calabi-Yau hypersurface $Y  = (f(x_0, \dots, x_{n-k+1}) + x_{n-k+2}^d + \dots + x_{n-m+1}^d =0 ) \subset \mathbb{P}^{n+1-m}$ is smooth.  
\end{proof}

We recover the well-known result on stability of a general hypersurface, originally shown in \cite{tian2012canonical}:

\begin{corollary}\label{cor:general-hypersurface}
    A general degree $d$ Fano hypersurface $X_d \subset \bP^{n+1}$ is K-polystable.  If $d \geq 3$, then $X$ is K-stable.
\end{corollary}

\begin{proof}
    If $d = 1$, $X_1 \cong \bP^n$ which is K-polystable.  If $d = 2$ and $n = 1$, then $X_2\subset \bP^2$ is $X_2 \cong \bP^1$ which is K-polystable.  If $d = 2$ and $n > 1$, the result follows from Corollary \ref{cor:hypersurface-cover} and induction as every smooth quadric can be written in the form of the Corollary.  If $d > 2$, then by Corollary \ref{cor:fermat-generalization}, there exists a K-stable hypersurface of degree $d$.  By openness of K-stability \cite[Theorem 4.5]{BLX19}, it follows that the general Fano hypersurface is K-stable.
\end{proof}

The remaining two corollaries in this section are generalizations of \cite[Cases (d), (e), pg. 3]{AGP06}. 

\begin{corollary}\label{cor:ci-base-case}
    Let $X$ be a Fano variety constructed as an $m:1$ cyclic cover of $\bP^n$ branched over a degree $d$ hypersurface $Y$, for $m\mid d$ and $m \ge 2$.  Assume in addition
    \begin{enumerate}
        \item if $d \le n$, $Y$ is K-polystable, or
        \item if $d > n$, $Y$ is GIT polystable (which holds automatically if $Y$ is smooth by Lemma \ref{lem:GITstability-hypersurfaces}). 
    \end{enumerate}
    Then, $X$ is K-polystable.
\end{corollary}

This includes, in particular, double covers of $\bP^n$ branched over a smooth hypersurface of even degree $d$ such that $\frac{n+1}{2} \leq \frac{d}{2} \le n$  generalizing \cite[Case (d), pg. 3]{AGP06}, and general weighted projective hypersurfaces $X_d \subset \bP(1^{n+1},a_{n+1})$ such that $a_{n+1} \mid d$ (in the notation of the corollary, $m = \frac{d}{a_{n+1}}$): if $Y = (f= 0) \subset \bP^n$ is a general Fano hypersurface of degree $d$, it is K-polystable by Corollary \ref{cor:general-hypersurface}, and the $m:1$ cover is given by $X = (f + x_{n+1}^m = 0) \subset \bP(1^{n+1}, a_{n+1})$.  Because $X$ is K-polystable, the general such $X_d$ is K-polystable.  

\begin{proof}
    This is Theorem \ref{thm:kstab-of-cover} applied to $Z = \bP^n$ and $Y \in |\calO_{\bP^n}(d)|$. 
\end{proof}

\begin{corollary}
    Let $X$ be a Fano variety constructed as an $m:1$ cyclic cover of a K-polystable Fano hypersurface $Z \subset \bP^{n+1}$ of degree $k$, branched over a hypersurface $Y \subset |\calO_{Z}(d)|$, for $m \ge 2$, $m \mid d$.  Assume in addition $Y$ is GIT polystable (which holds automatically if $Y$ is smooth by Lemma \ref{lem:GITstability-finiteAut} or Lemma \ref{lem:GIT-of-(2,d)ci}), and, if $d < n+2-k$, $Y$ is K-polystable.  Then, $X$ is K-polystable.
\end{corollary}

This includes, in particular, double covers of a smooth quadric $Q \subset \bP^{n+1}$ branched over a smooth hypersurface of even degree $d$ such that $\frac{n}{2} \leq \frac{d}{2} < n$  as in \cite[Case (e), pg. 3]{AGP06}. 

\begin{proof}
    This is Theorem \ref{thm:kstab-of-cover} applied to the given $Z$ and $Y$. 
\end{proof}

\section{K-stability of complete intersections with small Fano index}

In this section, we study the case of small Fano index.  The same results on complete intersections $X \subset \bP^{n+k}$ of Fano index 1 were obtained simultaneously by Yijue Hu, which will appear in forthcoming work. 

Using the machinery developed in the previous sections, we obtain the K-stability of general complete intersections of small index.

\begin{theorem}\label{thm:index1}
	Let $k \ge 1$ and $n \ge 2$.  Let $X\subset \mathbb P(1^{n+1},a_{n+1},\ldots,a_{n+k})$ be a Fano weighted complete intersection of type $(r_1,\ldots,r_k)$, where $r_i \ge 2$ and  $a_{n+i} \, \vert \, r_i$ for each $i$, with the following form: 
	\begin{equation*}
		X\colon \begin{cases}
			f_1+x_{n+1}^{r_1/{a_{n+1}}}=0\\
			f_2+x_{n+2}^{r_2/{a_{n+2}}}=0\\
			\quad \vdots \\
			f_k+x_{n+k}^{r_k/{a_{n+k}}}=0\\
		\end{cases} \subset \mathbb P(1^{n+1},a_{n+1},\ldots,a_{n+k})
	\end{equation*}
	such that $f_i \in \mathbb C[x_0,\ldots, x_n]$ is a homogeneous polynomials of degree $r_i$ for $1\leq i\leq k $ and  $\sum B_i$ is snc where $B_i$ is the divisor in $\mathbb P^n$ corresponding to the vanishing of $f_i$. Let $q:=\min\Big\{\frac{a_{n+1}}{r_1},\ldots,\frac{a_{n+k}}{r_k}\Big\}$. 
    \begin{enumerate}
        \item If the Fano index $\iota_X$ satisfies
	\begin{equation} \label{eq:smallindexineq}
	\iota_X < (n+1)q    
	\end{equation}
	then $X$ is K-polystable and is K-stable if $\mathrm{Aut}(X)$ is finite.
    \item If the Fano index $\iota_X$ satisfies
	\begin{equation} \label{eq:smallindexeq}
	\iota_X = (n+1)q    
	\end{equation}
	and each $B_i$ is a GIT-polystable hypersurface of $\mathbb P^n$, then $X$ is K-polystable and is K-stable if $\mathrm{Aut}(X)$ is finite.
    \end{enumerate}
In particular, if $X$ as above is smooth such that $\iota_X\leq (n+1)q$ and is not isomorphic to $\mathbb P^n$ or a quadric, then it is K-stable.
\end{theorem}

\begin{proof}
    We denote by $\mathbb P$ the weighted projective space $\mathbb P(1^{n+1},a_{n+1},\ldots,a_{n+k})$. Let $G=\prod_{i=1} ^k \mathbb Z/r_k\mathbb Z$ be a group acting on $\mathbb P$ via $\xi \cdot (x_0,\ldots,x_{n+k})=(x_0,\ldots,\xi_1 x_{n+1},\ldots \xi_kx_{n+k})
$ where $\xi_i \in \mathbb Z/r_i\mathbb Z$ is a primitive $r_i$-th root of unity. Then $G$ acts on $X$ and we have a quotient morphism $X\rightarrow X/G\simeq \mathbb P^n$ branched along the divisor $B=\sum_{i=1}^k \big( 1-\frac{a_{n+i}}{r_i}\big)B_i$. Moreover,
$$
\pi^*(K_{\mathbb P^n}+B) = \pi^*\bigg(K_{\mathbb P^n}+\sum_{i=1}^k \big( 1-\frac{a_{n+i}}{r_i}\big)B_i\bigg) =\bigg(-n-1- \sum_{i=1}^k a_{n+i} + \sum_{i=1}^k r_i\bigg)H_X  = K_X.
$$
So $\pi \colon X \rightarrow \mathbb P^n$ is Galois (See \cite[Definition~3.1]{liu2022equivariant}) and $(\mathbb P^n,B)$ is a log Fano pair by \cite[Proposition~5.20]{KM} or direct observation.

 Let $\Gamma :=\lambda B$ where $\lambda =\frac{n+1}{\sum_{i=1}^k r_i - \sum_{i=1}^k a_{n+i}}$. Without loss of generality we may assume $r_i\geq 2 a_{n+i}$ for each $i$ since if $r_j = a_{n+j}$ for some $1\leq j\leq k$ then $X$ is isomorphic to a complete intersection of codimension $k-1$. Therefore, $\sum_{i=1}^k r_i - \sum_{i=1}^k a_{n+i} \geq \sum_{i=1}^k a_{n+i} \geq k \geq 1 $. Recall that the Fano index of $X$ is $1 \leq \iota_X=\sum_{i=0}^{n+k}a_i-\sum_{i=1}^k r_i=n+1+\sum_{i=1}^ka_{n+i}-\sum_{i=1}^k r_i$. We conclude that $\lambda \geq 1+\frac{1}{\sum_{i=1}^kr_i - \sum_{i=1}^ka_{n+i}} > 1$. Since $\mathbb P^n$ is smooth and $B$ is an snc divisor by assumption then the Calabi-Yau pair $(\mathbb P^n,\Gamma)$ is klt (resp. lc) if and only if 
 \begin{equation*} \label{eq:klt}
 \lambda\Big(1-\frac{a_{n+j}}{r_{j}}\Big) < 1, \quad \forall \,\, 1\leq j\leq k  \quad  \quad \Bigg(\text{resp.} \,\,\,\,\,\lambda\Big(1-\frac{a_{n+j}}{r_{j}}\Big) \leq 1, \quad \forall \,\, 1\leq j\leq k \Bigg)  
\end{equation*}
which is equivalent to $\lambda(1-q) < 1$ (resp. $\lambda(1-q) \leq 1$) where $q:=\min\Big\{\frac{a_{n+1}}{r_1},\ldots,\frac{a_{n+k}}{r_k}\Big\}$. Since $1 \leq \iota_X=\sum_{i=0}^{n+k}a_i-\sum_{i=1}^k r_i$ so that $\sum_{i=1}^k r_i - \sum_{i=1}^k a_{n+i}=n+1-\iota_X$, the previous inequality becomes $\frac{n+1}{n+1-\iota_X} \cdot (1-q) < 1.
$ (resp. $\frac{n+1}{n+1-\iota_X} \cdot (1-q) \leq 1.
$)
In other words, we showed that the Calabi-Yau pair $(\mathbb P^n,\Gamma)$ is klt (resp. lc) if and only if 
\begin{equation} \label{eq:indexcond}
\iota_X < (n+1)q. \quad \Big(\text{resp.} \,\,\, \,\iota_X \leq (n+1)q\Big)
\end{equation}
In the klt case, the K-stability of $(\PP^n,cB)$ for all $c\in (0,\lambda)$ follows directly from  Proposition \ref{prop:polystabilityofpair}. Otherwise, in the lc case, we may assume that the $B_i$ are quadrics and the GIT-stability of $B$ is Lemma \ref{lem:sncquadrics}. So we can now apply Proposition \ref{prop:polystabilityofpair} to conclude similarly the K-stability of $(\PP^n,cB)$ for all $c\in (0,\lambda)$.

Hence, choosing $c=1$ shows that $(\mathbb P^n,B)$ is K-stable (resp. K-polystable). By \cite[Theorem~1.2]{liu2022equivariant} we conclude that $X$ is K-polystable. On the other hand, if  $\mathrm{Aut}(X)<\infty $,  it follows that $X$ is K-stable.

The last assertion follows from the fact that smooth complete intersections in well-formed weighted projective space have finite automorphism group with the exception of $\mathbb P^n$ and the quadric hypersurface. See \cite[Theorem~1.3]{autsmthWCI}.
\end{proof}

The index conditions are satisfied automatically for small index: 

\begin{corollary}\label{cor:index-one}
Let $k \ge 1$ and $n \ge 2$.  Let $X\subset \mathbb P(1^{n+1},a_{n+1},\ldots,a_{n+k})$ be a general Fano weighted complete intersection of type $(r_1,\ldots,r_k)$, where $r_i \ge 2$ and  $a_{n+i} \, \vert \, r_i$ for each $i$.  If $X$ has Fano index 1, then $X$ is K-polystable and K-stable if $\Aut(X)$ is finite.
\end{corollary}

\begin{proof}
    By Theorem \ref{thm:index1}, this holds if $1 < (n+1)q$ where $q:=\min\Big\{\frac{a_{n+1}}{r_1},\ldots,\frac{a_{n+k}}{r_k}\Big\}$.  To complete the proof, assume that $1 \ge (n+1)q$.  Up to relabeling, we may assume that $q=\frac{a_{n+1}}{r_1}$, so $r_1 \ge (n+1) a_{n+1}$.  Because $X$ was Fano, we have $r_1 + \dots + r_k \le n + a_{n+1} + \dots + a_{n+k}$, and therefore 
    \[ (n+1)a_{n+1} + r_2 + \dots + r_k \le r_1 + \dots + r_k \le n + a_{n+1} + \dots + a_{n+k}. \] Rearranging, we find 
    \[ (r_2 - a_{n+2}) + \dots + (r_k - a_{n+k}) \le n(1 - a_{n+1}).\]
    As the right side is non-positive and the left is non-negative, this is possible only if both are $0$, and hence $r_i = a_{n+i}$ for $i \ge 2$ and $1 = a_{n+1}$.  This implies $X$ is isomorphic to a hypersurface of degree $r_1$ in $\bP^{n+1}$, which is K-polystable by \cite{kentohyper}.
\end{proof}

\section{K-stability of the general complete intersection}\label{sec:generalcase}

The aim of this section is to prove the general complete intersection of arbitrary dimension and multi-degree is K-polystable, and in fact K-stable if it is not isomorphic to a $\bP^n$ or a quadric.  We will use induction on the dimension and codimension of the complete intersection, primarily making use of Theorem \ref{thm:kstab-of-cover}.  We will prove a more general result for weighted projective spaces and deduce the complete intersection result as a corollary.

For the induction, we first prove Theorem \ref{thm:general-case-induction} stated in the introduction. 

\begin{theorem}\label{thm:general-case-induction}
	    For integers $n \geq 2$ and $k \geq 2$, let $X \subset \mathbb{P}(1^{n+1}, a_{n+1}, \dots, a_{n+k})$ be a Fano weighted complete intersection of type $(r_1, \dots, r_k)$ such that $a_{n+i} \mid r_i$ and dimension $\dim X=n$ with equations
	\begin{equation*}
		X\colon \begin{cases}
			f_1=0\\
			f_2=0\\
			\quad \vdots \\
			f_{k-1}=0 \\
			f_k+x_{n+k}^{r_k/a_{n+k}}=0\\
		\end{cases} 
	\end{equation*}
	where $f_1, \dots ,f_k \in \mathbb{C} \left[ x_0, \dots , x_{n+k-1} \right]$ are a regular sequence of homogeneous polynomials of degree $\deg f_i = r_i$. 
	Call $Y \subset \mathbb{P}(1^{n+1}, a_{n+1}, \dots, a_{n+k-1})$ the complete intersection of type $(r_1, \dots, r_{k-1})$ defined by 
	\begin{equation*}
		Y \colon \begin{cases}
			f_1=0\\
			f_2=0\\
			\quad \vdots \\
			f_{k-1}=0 \\
		\end{cases} 
	\end{equation*}
	Call $B_Y$ the complete intersection $Y \cap (f_k=0) \subset \mathbb{P}(1^{n+1}, a_{n+1}, \dots, a_{n+k-1})$.  Assume $Y$ is K-polystable. If $B_Y$ is Fano, assume that $B_Y$ is K-polystable, and if $B_Y$ is not Fano, assume that the pair $(Y, \frac{n+1+\sum_{i=1}^{k-1} a_{n+i}-\sum_{i=1}^{k-1} r_i}{r_k}B_Y)$ is log canonical. Then, $X$ is K-polystable. If $\Aut(X)$ is finite, then $X$ is K-stable.
\end{theorem}

\begin{proof}
    Consider the map \[\pi: \mathbb{P}(1^{n+1}, a_{n+1},\dots,a_{n+k}) \to \mathbb{P}(1^{n+1}, a_{n+1},\dots,r_{k})\] defined by \[[x_0: \dots :x_{n+k-1}:x_{n+k}] \mapsto [x_0: \dots: x_{n+k-1}: x_{n+k}^{r_k/a_{n+k}}].\]  Denoting by $y_{n+k}$ the last coordinate, the image of $X$ under this map is defined by the equations 
    \begin{equation*}
        X\colon \begin{cases}
            f_1=0\\
            f_2=0\\
            \quad \vdots \\
            f_{k-1}=0 \\
            f_k+y_{n+k}=0\\
            \end{cases} 
    \end{equation*}
so the image of $X$ is the $Y \subset \mathbb{P}(1^{n+1}, a_{n+1},\dots,a_{n+k-1})$ where 
    \begin{equation*}
        Y \colon \begin{cases}
            f_1=0\\
            f_2=0\\
            \quad \vdots \\
            f_{k-1}=0 \\
            \end{cases} 
    \end{equation*}
    as above. By assumption, $Y$ is K-polystable.  

\textit{Case 1:} Assume $B_Y$ is not Fano. Then, the Fano index of $X$ must satisfy $\iota_X\le a_{n+k}$. Consider the log Calabi-Yau pair $(Y,\frac{\iota_Y}{r_k}B_Y)$, where $B_Y \in |\calO _Y(r_k)|$ and 
$$
\calO_Y(-K_Y) \cong \calO_Y\Big(n+1+\sum_{i=1}^{k-1} a_{n+i}-\sum_{i=1}^{k-1} r_i\Big)
=\calO_Y(\iota_X-a_{n+k}+r_k).
$$
Notice that $\iota_Y =\iota_X-a_{n+k}+r_k \le r_k$ since $\iota_X\leq a_{n+k}$ and thus $\frac{\iota_Y}{r_k} \le 1$.  By assumption, $(Y,\frac{\iota_Y}{r_k}B_Y)$ is log canonical. Moreover, if $Y$ is not a quadric hypersurface, then $B_Y$ is GIT polystable with respect to the action of $\Aut(Y)$ by Lemma \ref{lem:GITstability-finiteAut}. Otherwise the GIT stability of $B_Y$ follows from Lemma \ref{lem:GIT-of-(2,d)ci}. Hence, we can apply case (2) of Theorem \ref{thm:kstab-of-cover} to conclude that $X$ is K-polystable.

\textit{Case 2:} Assume $B_Y$ is Fano.  Then, by assumption, $B_Y$ is K-polystable. Furthermore, $B_Y \in |\calO_Y(r_k)|$ and $\calO(-K_Y) \cong \calO(n+1+\sum_{i=1}^{k-1} a_{n+i} - \sum_{i=1}^{k-1} r_i)$.  Therefore, by case (1) of Theorem \ref{thm:kstab-of-cover}, $X$ is K-polystable.
\end{proof}

Now, we prove the stability of the general complete intersection. 

\begin{theorem}\label{thm:general-case-weighted-ci}
    Let $X \subset \mathbb{P}(1^{n+1}, a_{n+1},\dots,a_{n+k})$ be a general $n$-dimensional Fano weighted complete intersection of type $(r_1, \dots, r_k)$ with $n \geq 1$ and $k\geq 1$.  Assume $a_{n+i} \mid r_i$ for each $1\le i \le k$.  Then, $X$ is K-polystable, and is K-stable if $\Aut(X)$ is finite. 
\end{theorem}

\begin{proof}
    We may assume that $r_1, \dots, r_k \ge 2$ by replacing $\bP(1^{n+1},a_{n+1}, \dots, a_{n+k})$ with a subspace defined by the vanishing of one coordinate function, and Proposition \ref{prop:smoothness-of-wted-ci}, we may assume that $X$ is smooth.  If $X \cong \bP^n$ or $X$ is a smooth quadric, then $X$ is K-polystable (see, e.g. Corollary \ref{cor:general-hypersurface}), so we assume that $X$ is not isomorphic to $\bP^n$ or a smooth quadric.  By openness of K-stability, it suffices to produce a single example of a K-stable complete intersection.
    
    We prove the result by a double induction on the dimension $n$ and the codimension $k$ of the complete intersection $X$.  If $(n,k) = (1,k)$, then $X$ is a smooth Fano curve, and therefore $X \cong \mathbb{P}^1$ is which K-polystable.  If $(n,k) = (n,1)$, then $X$ is a general smooth Fano weighted hypersurface, which is K-polystable by Corollary \ref{cor:ci-base-case}, and in fact K-stable by the assumption that $X$ is not isomorphic to $\bP^n$ or a smooth quadric.  To run the induction, we introduce the following notation.

    \begin{setup} \label{setup: X and Y index >=2}
    For integers $n \geq 2$ and $k \geq 2$, let $X \subset \mathbb{P}(1^{n+1}, a_{n+1},\dots,a_{n+k})$ be a Fano complete intersection of type $(r_1, \dots, r_k)$ with $a_{n+i} \mid r_i$ and dimension $\dim X=n$ with equations
    \begin{equation*}
        X\colon \begin{cases}
            f_1=0\\
            f_2=0\\
            \quad \vdots \\
            f_{k-1}=0 \\
            f_k+x_{n+k}^{r_k/a_{n+k}}=0\\
            \end{cases} 
    \end{equation*}
    where $f_1, \dots f_k \in \mathbb{C} \left[ x_0, \dots , x_{n+k-1} \right]$ are general homogeneous polynomials of degree $\deg f_i = r_i$. 
    Moreover, call $Y \subset \mathbb{P}(1^{n+1}, a_{n+1},\dots,a_{n+k-1})$ the complete intersection of type $(r_1, \dots, r_{k-1})$ defined by 
    \begin{equation*}
        Y \colon \begin{cases}
            f_1=0\\
            f_2=0\\
            \quad \vdots \\
            f_{k-1}=0 \\
            \end{cases} 
    \end{equation*}
    Call $B_k$ the hypersurface $B_k \coloneqq (f_k=0)$ in $\mathbb{P}(1^{n+1}, a_{n+1},\dots,a_{n+k-1})$ and denote by $B_Y = B_k|_Y = (f_1 = \dots = f_k = 0)$.  As the $f_i$ are general, $B_Y$ is a weighted complete intersection. Observe that, by moving one of the weights $1$ to the end of the list and writing $B_Y \subset \bP(1^{n}, a_{n+1},\dots,a_{n+k-1},1) =: \bP(1^n,b_{(n-1)+1}, \dots,b_{n-1+k})$, $B_Y$ is a complete intersection of dimension $n-1$ and of multidegree $(r_1, \dots, r_k)$ in $\bP(1^n,b_{(n-1)+1}, \dots,b_{n-1+k})$ such that $b_{(n-1)+i} \mid r_i$.
\end{setup}

Let $X \subset \mathbb{P}(1^{n+1}, a_{n+1},\dots,a_{n+k})$ be as in Setup \ref{setup: X and Y index >=2}, so $X$ has dimension $n$ and codimension $k$.  Assume by induction that the general weighted complete intersection of dimension $n-1$ and codimension $k$ in a weighted projective space $\mathbb{P}(1^{n}, b_{(n-1)+1},\dots,b_{(n-1)+k})$ defined by $k$ equations of degree $s_i$ such that $b_{(n-1)+i} \mid s_i$ is K-polystable.  Similarly, assume by induction that the general weighted complete intersection of dimension $n$ and codimension $k-1$ in $\mathbb{P}(1^{n+1}, c_{n+1},\dots,c_{n+(k-1)})$ defined by $k-1$ equations of degree $t_i$ such that $c_{n+i} \mid t_i$ is K-polystable.\footnote{This induction may be explained as follows: we induct on $n$, starting from the base case $(n,k) = (1,k)$, where the result holds for any $k$, and then assume for induction the result holds in dimension $n-1$ for any codimension $k$.  To prove this, we use induction on $k$ to show the result holds for $(n,k)$.  The base case of this induction is $(n,1)$, which holds in general, and the inductive hypothesis is then the result holds for $(n,k-1)$.  Therefore, from the two base cases that the result holds for $(n,1)$ and $(1,k)$ for any $n,k$, we may assume by induction the result holds for $(n-1,k)$ and $(n,k-1)$.} 

By the inductive hypothesis, because $Y$ as in Setup \ref{setup: X and Y index >=2} is general of dimension $n$ and codimension $k-1$, we assume $Y$ is K-polystable.  Because $B_Y$ as in Setup \ref{setup: X and Y index >=2} is general of dimension $n-1$ and codimension $k$, if it is Fano, we may assume by induction $B_Y$ is K-polystable, and if it is not Fano, we may assume it is smooth by Proposition \ref{prop:smoothness-of-wted-ci}.  By the computation in the proof of Theorem \ref{thm:general-case-induction}, as $B_Y$ is not Fano, $\frac{n+1+\sum_{i=1}^{k-1} a_{n+i}-\sum_{i=1}^{k-1} r_i}{r_k} \le 1$ and hence $(Y, \frac{n+1+\sum_{i=1}^{k-1} a_{n+i}-\sum_{i=1}^{k-1} r_i}{r_k}B_Y)$ is log canonical.  The hypotheses of Theorem \ref{thm:general-case-induction} are satisfied and $X$ is K-polystable. 

Finally, by openness of K-stability \cite[Theorem 4.5]{BLX19}, the general Fano weighted complete intersection of type $(r_1, \dots, r_k)$ in $\mathbb{P}(1^{n+1}, a_{n+1},\dots,a_{n+k})$ such that $a_{n+i} \mid r_i$ is K-polystable and K-stable if it is not isomorphic to $\bP^n$ or a quadric.  This proves Theorem \ref{thm:wted-version}.
\end{proof}

Letting $a_{n+1} = \dots = a_{n+k} = 1$ in the previous result, we obtain the K-stability of the general complete intersection, which proves Theorem \ref{thm: index >=2}.

\begin{corollary}
        Let $X \subset \mathbb{P}^{n+k}$ be a general $n$-dimensional Fano complete intersection of type $(r_1, \dots, r_k)$ with $n \geq 1$ and $k\geq 1$.  Assume $r_i \ge 2$ for each $i$, and if $\min r_i = 2$, then $k \ge 2$. Then, $X$ is K-stable. 
\end{corollary}


Finally, we conclude with several corollaries on K-stability of certain complete intersections.  The first is known by \cite{AGP06} but we provide an algebraic proof using the induction above.

\begin{corollary}\label{cor:int-of-2-quadrics}
    Let $X \subset \bP^{n+2}$ be any smooth Fano complete intersection of two quadrics, $n \ge 2$.  Then, $X$ is K-stable.
\end{corollary}

\begin{proof}
    Observe that a smooth complete intersection of two quadrics in $\bP^4$ is K-stable by \cite{MM90}.  We now assume $n \ge 3$ and proceed by induction on $n$. 
    
    Let $Q_1, Q_2 \subset \bP^{n+2}$ be two quadric hypersurfaces and suppose $X = Q_1 \cap Q_2$ is smooth.  By simultaneous diagonalization, up to switching $Q_1$ and $Q_2$ and replacing $Q_2$ with a generic element of the pencil $\langle Q_1, Q_2 \rangle$, we may assume $Q_1 = (f_1 = 0)$ and $Q_2 = (f_2 = 0)$ where 
    \[ f_1 = \sum_{i=0}^{n+2} a_i x_i^2 \qquad f_2 = \sum_{i=0}^{n+2}  x_i^2 \]
    using the assumption that $X$ is smooth to assume $f_2$ has full rank.

    By replacing $f_1$ with $f_1 - a_{n+2} f_2$, we may assume $a_{n+2} = 0$.  The smoothness assumption on $X$ then implies $a_i \ne 0$ for $i \ne {n+2}$.  Write $f_2 = f_2' + \sum_{i=0}^{n+1} x_i^2$.  Then,      \begin{equation*}
        X\colon \begin{cases}
            f_1=0\\
            f_2' + x_{n+2}^2=0\\
            \end{cases} 
    \end{equation*}
    where $f_1, f_2' \in \mathbb{C} \left[ x_0, \dots , x_{n+1} \right]$.

    With the notation from Theorem \ref{thm:general-case-induction}, $Y$ is a smooth quadric in $\bP^{n+1}$ and $B_Y$ is a smooth complete intersection of two quadrics in $\bP^{n+1}$.  By Corollary \ref{cor:general-hypersurface}, $Y$ is K-polystable, and as $B_Y$ is a smooth complete intersection of two quadrics of lower dimension, by the inductive hypothesis, $B_Y$ is K-polystable.  Then, by Theorem \ref{thm:general-case-induction}, $X$ is K-polystable.  As $X$ has a finite automorphism group by \cite{LW84} or \cite[Theorem 3.1]{benoist2013separation}, it is K-stable. 
\end{proof}



We can also prove that complete intersections defined by certain nets of quadrics in $\bP^7$ are K-stable, bootstrapping from the fact that every smooth complete intersection of three quadrics in $\bP^6$ is K-stable by \cite{AZ-Seshadri}. 

\begin{theorem}\label{three-quadrics}
	Let $X$ be smooth complete intersection of three quadrics $X = Q_1' \cap Q_2' \cap Q_3' \subset \bP^7$.  If $X$ can be written as $X = Q_1 \cap Q_2 \cap Q_3$ for some $Q_1, Q_2, Q_3$ in the net of quadrics $\langle Q_1', Q_2', Q_3' \rangle$ such that $Q_1$ and $Q_2$ have a common singular point $p$, then $X$ is K-stable.
\end{theorem}

We begin with a preliminary lemma. 

\begin{lemma}
	Let $X \subset \bP^n$ be a smooth complete intersection of three quadrics such that $X = Q_1 \cap Q_2 \cap Q_3$ and $Q_1, Q_2$ have a common singular point $p$.  Then, $p \notin Q_3$ and both $Q_1$ and $Q_2$ may be assumed to have rank $n-1$.
\end{lemma}

\begin{proof}
	If $p \in Q_3$, then $X$ would be singular by the Jacobian criterion, which contradicts the smoothness assumption.  Without loss of generality, we may replace $Q_1$ and $Q_2$ with any linearly independent elements of the pencil $\langle Q_1, Q_2\rangle $, all of which are necessarily singular at $p$.  We claim the smoothness assumption on $X$ forces the general element of this pencil to be a quadric of rank $n-1$.  Indeed, we may simultaneously diagonalize $Q_1, Q_2$ so they are given by equations of the form 
	\[ Q_1 = (a_0x_0^2 + \dots + a_{n-1}x_{n-1}^2 = 0)\]
	\[ Q_2 = (b_0x_0^2 + \dots + b_{n-1}x_{n-1}^2 = 0)\]
	where the coefficient of $x_n^2$ vanishes by assumption that $Q_1, Q_2$ are singular.  The only way every element of this pencil has rank $< n-1$ is if there exists some $0 \le i \le n-1$ such that $a_i = b_i  = 0$.  Without loss of generality, assume $i = n-1$.  Then, this implies $Q_1, Q_2$ and every element of the pencil they span are singular along the curve $x_0 = \dots = x_{n-2} = 0$, and in particular the intersection $Q_1 \cap Q_2$ is singular there.  Intersecting with $Q_3$, we see that $Q_1 \cap Q_2 \cap Q_3$ is singular along the intersection points of $Q_3$ with $x_0 = \dots = x_{n-2} = 0$.  
\end{proof}

Now, we prove the theorem. 

\begin{proof}
	Write $X = Q_1 \cap Q_2 \cap Q_3 \subset \bP^n$ where $Q_1, Q_2$ have a common singular point $p$.  By the previous lemma, we have $p \notin Q_3$ and may assume $Q_1, Q_2$ have rank $n-1$. 
	
	Without loss of generality, assume $p = [0:0: \dots : 0: 1]$.  Then, by completing the square in the equation of $Q_1$, we may write 
	\[ Q_1 = (f_1(x_0, \dots, x_{n-1}) = 0)\]
	\[ Q_2 = (f_2(x_0, \dots, x_{n-1}) = 0)\]
	\[ Q_3 = (f_3(x_0, \dots, x_{n-1}) + x_n^2 = 0)\]
	where $f_1, f_2$ define smooth quadrics in $\bP^{n-1}$ by assumption.  Denote by $\overline{Q}_i = (f_i = 0) \subset \bP^{n-1}$.
	
	If $\overline{Q}_1 \cap \overline{Q}_2$ were singular, then $Q_1 \cap Q_2$ (the cone over $\overline{Q}_1 \cap \overline{Q}_2$) would be singular along a curve, which would contradict the smoothness of $X = Q_1 \cap Q_2 \cap Q_3$.  Therefore, $\overline{Q}_1 \cap \overline{Q}_2$ is smooth, and hence K-stable by Corollary \ref{cor:int-of-2-quadrics}.   Similarly, the complete intersection $\overline{Q}_1 \cap \overline{Q}_2 \cap \overline{Q}_3 \subset \bP^6$ must be smooth (if $[x_0: \dots: x_{n-1}]$ were a singular point of this intersection, then $[x_0: \dots :x_{n-1}:0]$ would be a singular point of $X$, contradicting its smoothness), so by \cite{AZ-Seshadri}, is K-stable.  Because $X$ is a $2:1$ cover of the smooth complete intersection $\overline{Q}_1 \cap \overline{Q}_2$ branched over the smooth variety $\overline{Q}_1 \cap \overline{Q}_2 \cap \overline{Q}_3$, by Theorem \ref{thm:kstab-of-cover}, $X$ is K-stable. 
\end{proof}

While we do not show that \textit{every} smooth complete intersection of three quadrics in $\bP^7$ is K-stable, the previous result is inductive in the following sense: from the K-stability of every smooth complete intersection of three quadrics in $\bP^{n}$, the argument above shows that any complete intersection of three quadrics in $\bP^{n+1}$ of the form as in Theorem \ref{three-quadrics} is K-stable.

\bibliography{ref}
\bibliographystyle{alpha}

\end{document}